\documentclass[reqno,11pt]{amsart}
\usepackage{amsmath,amsthm,amscd,amssymb,graphicx,enumerate,latexsym}

\usepackage[raiselinks,colorlinks]{hyperref}
\hypersetup{citecolor=blue}

\numberwithin{equation}{section}

\theoremstyle{plain}
\newtheorem{thm}{Theorem}[section]
\newtheorem{lemma}[thm]{Lemma}

\theoremstyle{definition}

\theoremstyle{remark}
\newtheorem{remark}{Remark}[section]

\def\Re{\mathop{\rm Re}\nolimits}
\def\Im{\mathop{\rm Im}\nolimits}

\newcommand{\ul}{\underline}

\newcommand{\sgn}{\text{\rm{sgn}}}

\newcommand{\ess}{\text{\rm{ess}}}
\newcommand{\ac}{\text{\rm{ac}}}

\newcommand{\s}{\text{\rm{s}}}

\DeclareMathOperator*{\esssupp}{ess\,supp}

\DeclareMathOperator*{\supp}{supp}
\DeclareMathOperator*{\tr}{tr}
\DeclareMathOperator*{\Int}{Int}
\allowdisplaybreaks

\title[Square-summable variation]{Square-summable variation and absolutely continuous spectrum}
\author{Milivoje Lukic}
\date\today
\email{milivoje.lukic@rice.edu}


\begin{document}
\begin{abstract}
Recent results of Denisov~\cite{Denisov09} and Kaluzhny--Shamis~\cite{KaluzhnyShamis12} describe the absolutely continuous spectrum of Jacobi matrices with coefficients that obey an $\ell^2$ bounded variation condition with step $p$ and are asymptotically periodic. We extend these results to orthogonal polynomials on the unit circle. We also replace the asymptotic periodicity condition by the weaker condition of convergence to an isospectral torus and, for $p=1$ and $p=2$, we remove even that condition.
\end{abstract}

\maketitle

\section{Introduction}

Let $\mu$ be a Borel probability measure on the unit circle $\partial\mathbb{D}$, whose support is not a finite set. Orthonormalizing the sequence $1, z, z^2, \dots$ with respect to $\mu$ leads to the sequence of orthonormal polynomials $\varphi_n(z)$, $n= 0,1,2,\dots$. They, and the reversed polynomials
\begin{equation}\label{1.1}
\varphi^*_n(z) = z^n \overline{\varphi_n(1/\bar z)},
\end{equation}
obey the Szeg\H o recursion relation, which can be written in matrix form as
\[
\begin{pmatrix}
\varphi_{n+1}(z) \\
\varphi^*_{n+1}(z)
\end{pmatrix}
=A(\alpha_n,z) \begin{pmatrix}
\varphi_n(z) \\
\varphi^*_n(z)
\end{pmatrix},
\]
where $\alpha_n\in \mathbb{D}$ is called a Verblunsky coefficient and
\[
A(\alpha,z) = \frac 1{\sqrt{1-\lvert \alpha\rvert^2}} \begin{pmatrix} z & - \bar \alpha \\ - \alpha z & 1 \end{pmatrix}.
\]
By Verblunsky's theorem, this determines a 1-1 correspondence between the measure $\mu$ and its sequence of Verblunsky coefficients $\alpha = \{\alpha_n\}_{n=0}^\infty \in \mathbb{D}^\infty$. Conversely, one may start from the sequence $\alpha$ and construct a unitary  five-diagonal matrix, called a CMV matrix, whose canonical spectral measure is precisely $\mu$; see \cite{CMV,OPUC1} for details.

If $\mu$ has the Lebesgue decomposition
\[
d\mu = w(\theta) \frac{d\theta}{2\pi} + d\mu_\s,
\]
the main goal will be to describe the essential support of $w(\theta)$, defined as
\[
\Sigma_\ac(\mu) = \{ e^{i\theta} \in \partial\mathbb{D} \mid w(\theta) > 0\}
\]
(or, more precisely, the equivalence class up to sets of Lebesgue measure zero of this set). The topological support of the absolutely continuous part of $\mu$ is then well known to be
\[
\supp (w dx) = \overline{ \Sigma_\ac(\mu)  }^\ess,
\]
where $\bar B^\ess$ denotes the essential closure of $B$, i.e. the set of $e^{i\theta} \in \partial \mathbb{D}$ such that $\lvert \{ e^{i\phi}\in B \mid \phi \in (\theta-\epsilon, \theta+\epsilon)\}  \rvert >0$ for all $\epsilon >0$. This set is exactly the absolutely continuous spectrum of the corresponding CMV matrix; see \cite{GesztesyMakarovZinchenko08} for an expository discussion.

This paper focuses on Verblunsky coefficients such that for some $p \in \mathbb{N}$,
\begin{equation}\label{1.2}
\sum_{n=0}^\infty \lvert \alpha_{n+p} - \alpha_n \rvert^2 < \infty.
\end{equation}
A conjecture, made by Simon \cite[Conjecture 12.1.12]{OPUC2} for $p=1$ and by Last \cite{Last07} for discrete Schr\"odinger operators, postulates that  \eqref{1.2} together with
\[
\lim_{n\to \infty} \alpha_n = 0
\]
 implies that $\esssupp w = \partial \mathbb{D}$. All the previously known results are in the setting of Jacobi matrices: a significant partial result was shown by Kupin~\cite{Kupin05}, and the full result for discrete Schr\"odinger operators was proved by Denisov~\cite{Denisov09} (who also proved the result for continuum Schr\"odinger operators \cite{Denisov02}). The method of \cite{Denisov09} was generalized by Kaluzhny--Shamis~\cite{KaluzhnyShamis12} to asymptotically periodic Jacobi matrices, with the result that the a.c.\ spectrum is equal to the essential spectrum for the limiting periodic sequence. A later version of the conjecture, by Breuer--Last--Simon~\cite[Conjecture 9.5]{BreuerLastSimon10}, concerns the situation when the asymptotic periodicity condition is removed. Here a sequence $\alpha$ is asymptotically periodic if for every $n\in \mathbb{N}_0$, the limit 
\begin{equation}\label{1.3}
\lim_{k\to\infty} \alpha_{kp+n} = \beta_n
\end{equation}
exists.

In this paper, we extend the method and results of \cite{Denisov09,KaluzhnyShamis12} to orthogonal polynomials on the unit circle. We also generalize those results, relaxing the condition \eqref{1.3}, which we consider the main contribution of this paper. The corresponding results for Jacobi matrices will be discussed in a forthcoming joint paper with Yoram Last~\cite{LastLukic}.

To motivate our goal of relaxing the condition \eqref{1.3}, note that existence of the limit in \eqref{1.3} does not follow from \eqref{1.2}; rather, it is an additional technical assumption. In contrast, the corresponding statement for the $\ell^1$ variation condition
\begin{equation}\label{1.4}
\sum_{n=0}^\infty \lvert \alpha_{n+p} - \alpha_n \rvert < \infty
\end{equation}
instead of \eqref{1.2} was proved by Golinskii--Nevai~\cite{GolinskiiNevai01}, who used the condition \eqref{1.3}; however, unlike \eqref{1.2}, \eqref{1.4} implies existence of the limit \eqref{1.3}, so in that setting, \eqref{1.3} was not an additional assumption.

In the results that follow, it will be convenient to assume
\begin{equation}\label{1.5}
\sup_{n\to\infty} \lvert \alpha_n \rvert < 1.
\end{equation}
There is no loss in this assumption, since by Rakhmanov's lemma  \cite{Rakhmanov83}, \cite[Theorem 4.3.4]{OPUC2},  $\sup_{n\to\infty} \lvert \alpha_n \rvert = 1$ implies $\esssupp w = \emptyset$.

For small $p$, we can describe $\esssupp w$ without any convergence condition:

\begin{thm}\label{T1.1}
Let \eqref{1.2} hold for $p=1$ and assume \eqref{1.5}. Then
\begin{equation}\label{1.6}
\Sigma_\ac(\mu) = \{e^{i\theta}\mid \theta \in [2 \arcsin A, 2\pi - 2 \arcsin A]\},
\end{equation}
where $A = \limsup_{n\to\infty} \lvert\alpha_n \rvert$. Moreover, for any closed arc $I \subset \Int(\Sigma_\ac(\mu) \setminus \{1\})$,
\begin{equation}\label{1.7}
\int_I \log w(\theta) \frac{d\theta}{2\pi} > -\infty.
\end{equation}
\end{thm}


\begin{thm}\label{T1.2}
Let \eqref{1.2} hold for $p=2$ and assume \eqref{1.5}. Then
\begin{equation}\label{1.8}
\Sigma_\ac(\mu) = \left\{ e^{i\theta} \in \partial \mathbb{D} \Bigm\vert - A_+ < \cos \theta <  A_- \right\} ,
\end{equation}
where
\[
A_\pm  = \liminf_{m\to\infty} \left( \rho_{2m}\rho_{2m+1}  \pm \Re(\alpha_{2m} \bar \alpha_{2m+1}) \right).
\]
Moreover, \eqref{1.7} holds for any closed arc $I \subset \Int (\Sigma_\ac(\mu) \setminus  \{-1, 1\})$.
\end{thm}

\begin{remark}
It is possible to have $-A_+ \ge A_-$, and in that case, \eqref{1.8} is the empty set. Otherwise, it is a union of two arcs symmetric about $\mathbb{R}$.
\end{remark}

\begin{remark}
The preceding theorems sometimes yield arcs with purely singular spectrum. In the setting of Theorem~\ref{T1.1}, a result of Last--Simon \cite[Theorem 3.1]{LastSimon06} for essential spectra of right limits implies that the essential spectrum is
\[
\{e^{i\theta}\mid \theta \in [2 \arcsin \ul A, 2\pi - 2 \arcsin \ul A]\}
\]
where $\ul A = \liminf_{n\to\infty} \lvert \alpha_n \rvert$. If $\ul A < A$, this is strictly greater than the set \eqref{1.6}, so the complement supports a purely singular part of the measure.
\end{remark}

\begin{remark}\label{R1.4}
The condition $1 \notin I$ in Theorem~\ref{T1.1} and $-1,1\notin I$ in Theorem~\ref{T1.2} is not an artifact of the method, but a real phenomenon. This may be seen from the Szeg\H o theorem for the unit circle, 
\[
\alpha \in \ell^2 \quad\Leftrightarrow\quad \int_{\partial \mathbb{D}} \log w(\theta) \frac{d\theta}{2\pi} > -\infty.
\]
In the context of Theorem~\ref{T1.1}, this implies that if $\sum_{n=0}^\infty \lvert \alpha_{n+1} - \alpha_n \rvert^2 < \infty$ and $\lim_{n\to\infty} \alpha_n =  0$ but $\alpha \notin \ell^2$, \eqref{1.7} holds for closed arcs with $1\notin I$, so it must fail for all arcs with $1\in \Int I$, even though $\esssupp w = \partial\mathbb{D}$.

Another point of view is provided by a higher order Szeg\H o theorem due to Simon~\cite[Section 2.8]{Rice},
\[
\alpha \in \ell^4, \quad \sum_{n=0}^\infty \lvert \alpha_{n+1} - \alpha_n \rvert^2 < \infty \quad\Leftrightarrow\quad \int_{\partial \mathbb{D}} (1-\cos \theta) \log w(\theta) \frac{d\theta}{2\pi} > -\infty.
\]
Thus, for $\alpha \in \ell^4 \setminus \ell^2$ and $1\in \Int I$, \eqref{1.7} fails but a weighted condition holds.

Similarly, the necessity of singling out $-1,1$ in Theorem~\ref{T1.2} can be seen from Szeg\H o's theorem and a higher order Szeg\H o theorem of Simon--Zlato\v s~\cite{SimonZlatos05},
\[
\alpha \in \ell^4, \quad \sum_{n=0}^\infty \lvert \alpha_{n+2} - \alpha_n \rvert^2 < \infty \quad\Leftrightarrow\quad \int_{\partial \mathbb{D}} (1-\cos^2\theta) \log w(\theta) \frac{d\theta}{2\pi} > -\infty.
\]
For more on higher order Szeg\H o theorems, see \cite[Section 2.8]{Rice}, \cite{LaptevNabokoSafronov03}, \cite{SimonZlatos05}, \cite{GolinskiiZlatos07}, \cite{Lukic7}.
\end{remark}

As we will see later, the cases $p=1$ and $p=2$ are special because for those values of $p$, closed gaps of $p$-periodic sequences can only occur at $p$-th roots of unity. For larger values of $p$, to exactly describe $\esssupp w$, we will assume convergence to an isospectral torus. Last--Simon~\cite{LastSimon06} and Damanik--Killip--Simon~\cite{DamanikKillipSimon10} analyzed perturbations of periodic Jacobi and CMV matrices and their work shows that convergence to an isospectral torus, rather than asymptotic periodicity, is the natural generalization of decaying perturbations of the free case. We now review the necessary definitions. 

For $m\in \mathbb{N}_0$, define $S^m \alpha = \{\alpha_{n+m}\}_{n=0}^\infty$. A sequence $\alpha^{(r)} = \{\alpha^{(r)}_n\}_{n=-\infty}^\infty \in \mathbb{D}^\infty$ is called a right limit of $\alpha$ if there is a sequence $n_j \in\mathbb{Z}$, $n_j \to +\infty$, such that $S^{n_j}\alpha$ converges pointwise to $\alpha^{(r)}$, i.e.\ for all $n\in \mathbb{Z}$,
\[
\lim_{j\to \infty} \alpha_{n+n_j} = \alpha^{(r)}_{n}.
\]
When \eqref{1.5} holds, a compactness argument shows that $\alpha$ has at least one right limit; the same argument shows that for every sequence $n_j \to +\infty$ there exists a pointwise convergent subsequence.

The condition \eqref{1.2} implies
\begin{equation}\label{1.9}
\lim_{n\to \infty} \lvert \alpha_{n+p} - \alpha_n \rvert =  0,
\end{equation}
which implies that all right limits of $\alpha$ are $p$-periodic since
\[
\alpha^{(r)}_{n+p} - \alpha^{(r)}_n =\lim_{j\to\infty} ( \alpha_{n+p+n_j} - \alpha_{n+n_j}) = 0.
\]
If $\{\gamma_n\}_{n=0}^\infty \subset \mathbb{D}^\infty$ is $p$-periodic, its discriminant is defined as
\begin{equation}\label{1.10}
\Delta(z) = z^{-p/2} \tr \left( A(\gamma_{p-1},z) A(\gamma_{p-2},z) \dots A(\gamma_0,z)  \right).
\end{equation}
For odd values of $p$, this has an ambiguity in the choice of branch of $z^{-p/2}$; this does not affect the statements below. It is well known \cite[Chapter 11]{OPUC2}  that the CMV matrix corresponding to Verblunsky coefficients $\gamma$ has essential spectrum
\[
e = \{ z \in \partial \mathbb{D} \mid \Delta(z) \in [-2,2] \}
\]
and that this set is a union of $p$ arcs on $\partial\mathbb{D}$.

The isospectral torus of $e$, denoted $\mathcal{T}_e$, is the set of all $p$-periodic sequences whose essential spectrum is equal to $e$. It is known that this set is generically a $p$-dimensional torus and that all elements of the isospectral torus have the same discriminant, which we will denote by $\Delta_e(z)$.

To define convergence to an isospectral torus, we need a metric on $\mathbb{D}^\infty$,
\[
d(\alpha,\gamma) = \sum_{n=0}^\infty e^{-n} \lvert \alpha_n - \gamma_n \rvert,
\]
which has the property that convergence in $d$ is equivalent to pointwise convergence. Then $\alpha$ is said to converge to $\mathcal{T}_e$ if and only if
\[
\lim_{m\to\infty} d( S^m\alpha, \mathcal{T}_e) =0.
\]
This is equivalent to saying that every accumulation point of $\{S^m\alpha\}_{m=0}^\infty$ lies on $\mathcal{T}_e$. Since accumulation points of $S^m\alpha$ are precisely right limits,  $\alpha$ converges to the isospectral torus $\mathcal T_e$ if and only if all of its right limits lie on $\mathcal T_e$.

By \cite{LastSimon06}, convergence to the isospectral torus $\mathcal{T}_e$ implies  $\esssupp \mu = e$. With our square-summable variation condition \eqref{1.2}, we can say the same of $\esssupp w$:

\begin{thm}\label{T1.3}
Let \eqref{1.2} hold for some $p\in\mathbb{N}$ and assume \eqref{1.5}. If $\{\alpha_n\}_{n=0}^\infty$ converges to the isospectral torus $\mathcal T_e$, then
\begin{equation}\label{1.11}
\Sigma_\ac(\mu) = e.
\end{equation}
Moreover, \eqref{1.7} holds for any closed arc $I \subset e$ such that $\lvert\Delta_e(z)\rvert < 2$ for all $z\in I$.
\end{thm}

All the above theorems will easily follow from our main result, an existence result for a.c.\ spectrum. This result does not require any convergence condition, so right limits will in general have different discriminants. We therefore define, as the supremum over all right limits of $\alpha$,
\begin{equation}\label{1.12}
L(z) = \sup_{(r)} \lvert \Delta^{(r)}(z) \rvert.
\end{equation}
In Lemma~\ref{L3.2} below, we will see that \eqref{1.5} implies that $L(z)$ is finite, that the $\sup$ is really a $\max$ and that $L(z)$ is continuous.

\begin{thm}\label{T1.4}
Let \eqref{1.2} hold for some $p\in\mathbb{N}$ and assume \eqref{1.5}. If $I \subset \partial\mathbb{D}$ is a closed arc such that
\begin{equation}\label{1.13}
\max_{z\in I} L(z) < 2,
\end{equation}
then \eqref{1.7} holds. Thus,
\begin{equation}\label{1.14b}
\{ z\in \partial \mathbb{D} \mid L(z) < 2 \} \subset \Sigma_\ac(\mu) \subset \{ z \in \partial\mathbb{D} \mid L(z) \le 2\}
\end{equation}
and
\begin{equation}\label{1.14}
\overline{\{ z\in \partial \mathbb{D} \mid L(z) < 2 \}} \subset \supp (w dx) \subset \overline{\{ z \in \partial\mathbb{D} \mid L(z) \le 2\}}^\ess.
\end{equation}
\end{thm}

The proof of the lower estimate for $\Sigma_\ac(\mu)$ will take up most of the paper, whereas the upper estimate will be an immediate corollary of a Last--Simon result for a.c.\ spectra of right limits.

Most of the paper will be dedicated to proving Theorem~\ref{T1.4}. Section \ref{S2} reviews well known properties of $p$-step transfer matrices and modifies them in a way which will be needed later. Section~\ref{S3} establishes various uniform estimates, which are needed in place of convergence. These estimates are used in Section~\ref{S4} to uniformly diagonalize the transfer matrices. Section~\ref{S5} introduces weak approximants for $\mu$ and relates their absolutely continuous parts to certain Weyl solutions. Section~\ref{S6} completes the proof of Theorem~\ref{T1.4}, using the method of \cite{Denisov09,KaluzhnyShamis12} with necessary modifications. Section~\ref{S7} uses Theorem~\ref{T1.4} to prove Theorems~\ref{T1.1}, \ref{T1.2}, \ref{T1.3}.

\section{$p$-step transfer matrices}\label{S2}

Let $\gamma_0, \dots, \gamma_{p-1} \in\mathbb{D}$, and let us define a $p$-step transfer matrix and its discriminant by
\begin{align*}
\Phi(z) & = A(\gamma_{p-1},z) A(\gamma_{p-2},z) \dots A(\gamma_{0},z) \\
\Delta(z) & = z^{-p/2} \tr \Phi(z)
\end{align*}
The sign ambiguity that arises for odd $p$ can be dealt with in any of several standard ways, such as sieving~\cite[Example 1.6.14]{OPUC1} or treating $\Delta$ as a function of $z^{1/2}$ or as a two-valued function. Our analysis will work on a fixed arc $I$, on which we can fix a branch of $z^{p/2}$ throughout the proof.

\begin{thm}[{\cite[Sections 11.1--11.2]{OPUC2}}]  \phantomsection \label{T2.1}
\begin{enumerate}[(i)]
\item $\det \Phi(z) = z^p$;
\item $\Delta$ is analytic in $\mathbb{C}\setminus\{0\}$;
\item $z \in \partial \mathbb{D}$ implies  $\Delta(z) \in \mathbb{R}$ and $iz \Delta'(z) \in \mathbb{R}$;
\item $\Delta(z) \in [-2,2]$ implies $z\in \partial\mathbb{D}$;
\item $\Delta(z) \in (-2,2)$ implies $\Delta'(z)\neq 0$.
\end{enumerate}
\end{thm}

These statements are usually made in the context of $p$-periodic Verblunsky coefficients, where $\Delta(z)$ is precisely the discriminant of the corresponding measure (compare with \eqref{1.10}). However, they can be viewed as merely facts about the $p$-step transfer matrix $\Phi(z)$.

Rather than working directly with $\Phi(z)$, we will alter it slightly. Let
\[
M = \frac 1{\sqrt 2} \begin{pmatrix} 1 & 1 \\ 1 & -1 \end{pmatrix}.
\]
Then $M=M^{-1} = M^*$. We introduce $\tilde \Phi(z)$ and its entries $a(z), b(z), c(z), d(z)$ by
\begin{equation}\label{2.1}
\tilde\Phi(z) = z^{-p/2} M \Phi(z) M  = \begin{pmatrix} a(z) & b(z) \\  c(z) & d(z) \end{pmatrix}.
\end{equation}
This has several useful properties, listed in the following theorem.

\begin{thm}  \phantomsection \label{T2.2} 
\begin{enumerate}[(i)]
\item $\det \tilde\Phi(z) = 1$;
\item $ \tr \tilde\Phi(z) = \Delta(z)$;
\item if $\lvert z \rvert = 1$, then $a(z), i b(z), i c(z), d(z) \in \mathbb{R}$;
\item if $\Delta(z) \in (-2,2)$, then $c(z) \neq 0$.
\end{enumerate}
\end{thm}

\begin{proof}
(i) and (ii) follow from Theorem~\ref{T2.1}(i) and cyclicity of trace.

To prove (iii), denote by $\varphi_p(z)$ and $\psi_p(z)$ the orthogonal and second kind orthogonal polynomials. It is known that
\[
\Phi = \tfrac 12 \begin{pmatrix} \varphi_p + \psi_p & \varphi_p - \psi_p \\   \varphi_p^* - \psi^*_p & \varphi_p^*  + \psi^*_p \end{pmatrix}
\]
so \eqref{2.1} implies
\[
\tilde\Phi = \tfrac 12 z^{-p/2}  \begin{pmatrix} \varphi_p + \varphi_p^* & \psi_p - \psi^*_p \\   \varphi_p - \varphi^*_p & \psi_p + \psi^*_p \end{pmatrix}.
\]
If $\lvert z \rvert =1$, \eqref{1.1} implies $\varphi_p^*(z) = z^p \overline{\varphi_p(z)}$. Thus,
\[
c(z) = \tfrac 12 z^{-p/2} (\varphi_p(z) - z^p \overline{\varphi_p(z)}) = i \Im ( z^{-p/2} \varphi_p(z) ).
\]
Claims for $a(z), b(z), d(z)$ are proved analogously.

(iv) is just \cite[Theorem 11.3.1]{OPUC2} in disguise.
\end{proof}

This simple trick of conjugating by $M$ does not seem to be present in the literature; however, it has the useful properties (iii) and (iv) above. While (iii) will be convenient in several places, (iv) will be crucial to our diagonalization procedure in Section~\ref{S4}.

\section{Estimates on transfer matrices}\label{S3}

We define the $p$-step transfer matrix between $mp$ and $(m+1)p$ and its rescaled trace,
\begin{align*}
\Phi_m(z) & = A(\alpha_{(m+1)p-1},z) A(\alpha_{(m+1)p-2},z) \dots A(\alpha_{mp},z) \\
\Delta_m(z) & = z^{-p/2} \tr \Phi_m(z)
\end{align*}
Following \eqref{2.1}, we also introduce $\tilde \Phi_m(z)$ and $a_m(z), b_m(z), c_m(z), d_m(z)$ by
\[
\tilde\Phi_m(z) = z^{-p/2} M \Phi_m(z) M  = \begin{pmatrix} a_m(z) & b_m(z) \\  c_m(z) & d_m(z) \end{pmatrix}.
\]

In this section, we make some preliminary observations about the $\tilde\Phi_m(z)$ and relate them to $L(z)$. They are mostly uniformness statements, necessary because we don't assume that $\tilde\Phi_m(z)$ converge and cannot apply local arguments around the limit.

We begin with a preliminary observation: although the notation $\Phi_m(z)$ is convenient, we will also find it useful to think about $\Phi_m(z)$ as a function of $\alpha_{mp},\alpha_{mp+1},\dots,\alpha_{(m+1)p-1}$ and $z$, with no $m$-dependence except through the $\alpha$'s. The same holds for $\tilde\Phi_m(z)$, its entries, and some functions to be introduced later.

\begin{lemma}\label{L3.1}
$\Phi_m(z)$ is an analytic function of real and imaginary parts of $\alpha_{mp}, \alpha_{mp+1}, \dots, \alpha_{(m+1)p-1} \in \mathbb{D}$ and an analytic function of $z\in \mathbb{C} \setminus \{0\}$. The same is true of $\tilde\Phi_m(z)$, $a_m(z)$, $b_m(z)$, $c_m(z)$, $d_m(z)$ and $\Delta_m(z)$.

For any such function $f_m(z)$, if \eqref{1.5} holds, then for any $R<\infty$, there is a constant $C< \infty$ such that  for all $m\ge 0$ and $z \in \overline{\mathbb{D}}_R \setminus \mathbb{D}_{1/R}$,
\begin{align}
\lvert f_m(z) \rvert  & \le C, \label{3.1}  \\
\lvert f_{m+1}(z) - f_m(z) \rvert  & \le C \sum_{k=0}^{p-1} \lvert \alpha_{(m+1)p+k} - \alpha_{mp+k} \rvert.\label{3.2}
\end{align}
In particular, if \eqref{1.2} also holds, then
\begin{equation}\label{3.3}
\sum_{m=0}^\infty \lvert f_{m+1}(z) - f_m(z) \rvert^2 < \infty.
\end{equation}
\end{lemma}

\begin{proof}
The entries of $A(\alpha_{mp+k},z)$ have the listed analyticity properties. Thus, so do entries of their products $\Phi_m(z)$, and by their definitions, so do the other functions listed.

By \eqref{1.5}, we are working with parameters $z\in  \overline{\mathbb{D}}_R \setminus \mathbb{D}_{1/R}$ and $\alpha_{mp+k} \in \overline{\mathbb{D}}_r$, with $r = \sup_n \lvert \alpha_n \rvert < 1$. Compactness of this set of parameters, together with analyticity of  $f_m(z)$, implies \eqref{3.1} and implies that the partial derivatives of $f$ in $\Re \alpha_{mp+k}$ and $\Im\alpha_{mp+k}$ are bounded. Boundedness of these partial derivatives implies \eqref{3.2} by the mean value theorem. \eqref{3.3} follows immediately from \eqref{3.2} and \eqref{1.2}.
\end{proof}

The following lemma relates $L(z)$ to the $\Delta_m(z)$ and establishes its properties.

\begin{lemma}\label{L3.2}
Assume that \eqref{1.5} and  \eqref{1.9} hold.  Then for all $z\in\mathbb{C}\setminus\{0\}$, $L(z)$ is a finite number,
\begin{equation}\label{3.4}
L(z) = \limsup_{m\to\infty} \lvert \Delta_m(z) \rvert,
\end{equation}
and the $\sup$ in \eqref{1.12} is actually a $\max$ (i.e.\ for every $z\neq 0$ there is a right limit for which $\lvert \Delta^{(r)}(z) \rvert = L(z)$). Moreover, for any $R<\infty$, $L(z)$ is Lipschitz continuous on ${\{z \in \mathbb{C} \mid R^{-1}  \le \lvert z\rvert \le R \} }$.
 \end{lemma}
 
 \begin{proof}
Let us define
\[
{\mathcal L}(z) =  \limsup_{m\to\infty} \lvert \Delta_m(z) \rvert.
\]
This quantity is finite for $z\in\mathbb{C}\setminus \{0\}$ by \eqref{3.1} applied to $\Delta_m(z)$.

Let $\Delta^{(r)}$ be the discriminant for the right limit $\alpha^{(r)}$ corresponding to indices $\{n_j\}_{j=1}^\infty$. By passing to a subsequence, make the sequence constant modulo $p$, i.e.\ $n_j = m_j p + q$; this is possible for some choice of $q \in \{0,1,\dots, p-1\}$. 
Cyclicity of trace, together with \eqref{1.10} and $p$-periodicity of $\alpha^{(r)}$, gives
\[
\lvert \Delta^{(r)}(z) \rvert  = \left\lvert z^{-p/2} \tr \left(  A(\alpha^{(r)}_{2p-q-1},z) A(\alpha^{(r)}_{2p-q-2},z) \dots A(\alpha^{(r)}_{p-q},z)   \right) \right\rvert.
\]
Thus, by $\alpha^{(r)}_n = \lim_{j\to\infty} \alpha_{n_j+n}$ and continuity of $A(\alpha,z)$ in $\alpha\in \mathbb{D}$,
\[
\lvert \Delta^{(r)}(z) \rvert  = \lim_{j\to \infty} \lvert \Delta_{m_j+1}(z) \rvert  \le {\mathcal L}(z).
\]
Since this holds for every right limit, we conclude $L(z)  \le {\mathcal L}(z)$.

For the opposite inequality, fix $z$ and let $m_k$ be a sequence of integers with $\lim_{k\to\infty} \lvert \Delta_{m_k}(z) \rvert = {\mathcal L}(z)$. By \eqref{1.5} and compactness,  a subsequence of  $\{m_k p\}_{k=1}^\infty$ gives rise to a right limit $\alpha^{(r)}$; for this right limit, $\lvert \Delta^{(r)}(z)  \rvert  = {\mathcal L}(z)$. This shows that $L(z) = {\mathcal L}(z) < \infty$ and that the $\sup$ in \eqref{1.12} is a $\max$.

Denote $r = \sup_n \lvert \alpha_n \rvert <1$. By Lemma~\ref{L3.1}, $\Delta_m(z)$ is an analytic function of real and imaginary parts of $\alpha_{mp}, \dots, \alpha_{(m+1)p-1} \in \overline{\mathbb D}_r$ and of $z \in \overline{\mathbb D}_R \setminus \mathbb{D}_{1/R}$. Since this set of parameters is compact, we conclude that the $\Delta_m(z)$ are uniformly Lipschitz continuous in $z \in \overline{\mathbb D}_R \setminus \mathbb{D}_{1/R}$. As the $\limsup$ of uniformly Lipschitz continuous functions, $L(z)$ is also Lipschitz continuous.
\end{proof}

\begin{remark}
If $\limsup_{n\to\infty} \lvert \alpha_n\rvert = 1$, one may be inclined to define $L(z)$ by \eqref{3.4}. However, some of the above properties would no longer be true. For instance, for $p=1$, $\lvert \Delta_m(z) \rvert = \frac {\lvert z+1 \rvert}{\sqrt{1-\lvert \alpha_m\rvert^2}} $, so $\limsup_{n\to\infty} \lvert \alpha_n \rvert =1$ would imply
\[
L(z) = \begin{cases} 0 & z=-1 \\ +\infty & z\neq -1 \end{cases}
\]
which is no longer finite or continuous.
\end{remark}

\begin{lemma}\label{L3.3}
Assume \eqref{1.5} and  \eqref{1.9}  and let $I\subset\partial\mathbb{D}$ be a closed arc such that \eqref{1.13} holds. Then there exist $m_0\in \mathbb{N}_0$, $s, t \in \{-1,+1\}$, $\epsilon\in (0,1)$ and $C>0$ such that for all $m\ge m_0$ and $z\in \Omega$,
\begin{align}
\lvert \Delta_m(z) \rvert & \le 2 - C \label{3.5} \\
s \Im \left( z \Delta'_m(z) \right) &  \ge C \label{3.6} \\ 
C \le t \Im \left( c_m(z) \right) & \le \lvert c_m(z) \rvert   \le C^{-1} \label{3.7}
\end{align}
where
\begin{equation}\label{3.8}
\Omega = \{ r e^{i\theta}\mid e^{i\theta} \in I, r \in [1-\epsilon,1]\}.
\end{equation}
\end{lemma}

\begin{proof}
The upper bound for $\lvert c_m(z)\rvert$ follows from Lemma~\ref{L3.1}. For the other estimates, it suffices to find $m_0, s, t, C$ such that they are true for $\Omega = I$; by uniform Lipschitz continuity of $\Delta_m$, $\Delta'_m$, $c_m$, the estimates will then, with a change of $C$, also hold on the set $\Omega$ given by \eqref{3.8} for a small enough $\epsilon>0$. Therefore, in the remainder of this proof, we work with $\Omega = I$.

To prove \eqref{3.5}, assume, on the contrary, that there are sequences $m_k \to \infty$, $z_k \in I$ with $\lvert \Delta_{m_k}(z_k) \rvert \ge 2$. By compactness of $I$ we may pass to a subsequence such that $z_k \to z_\infty \in I$; since the $\Delta_m$ are uniformly Lipschitz continuous, this implies $L(z_\infty) \ge 2$, which is a contradiction with \eqref{3.4}.

To prove \eqref{3.6}, let us first prove that
\begin{equation}\label{3.9}
\inf_{m\ge m_0} \min_{z\in I}  \lvert \Delta'_m(z) \vert > 0
\end{equation}
for large enough $m_0$. If this was false, there would exist sequences $m_k \to \infty$, $z_k \in I$ with $ \Delta'_{m_k}(z_k)  \to 0$. Passing to a subsequence with $z_k \to z_\infty$, since $\Delta'_m$ are uniformly Lipschitz continuous, gives $\lim_{k\to\infty} \Delta'_{m_k}(z_\infty) = 0$. By compactness, we may pass to a further subsequence so that $S^{m_kp} \alpha$ converges pointwise to a right limit $\alpha^{(r)}$. For that right limit, $\Delta^{(r)}{}'(z_\infty)=0$ but $\Delta^{(r)}(z_\infty) < 2$. This contradicts Theorem~\ref{T2.1}(v), proving \eqref{3.9}.

By \eqref{3.5} and Theorem~\ref{T2.1}, we know that
\[
iz \Delta_m'(z) = \frac{\partial }{\partial \theta} \Delta_m(z)
\]
is real and nonzero for $z = e^{i\theta} \in I$. Denote by $\delta$ the $\inf$ in \eqref{3.9}. A change of sign of $i z \Delta_m'(z)$ between $m$ and $m+1$ then requires
\begin{equation}\label{3.10}
\left \lvert iz  \Delta_{m+1}'(z) - iz  \Delta_m'(z) \right\rvert \ge 2\delta.
\end{equation}
The inequality \eqref{3.2} applied to $\Delta_m'(z)$, together with \eqref{1.9}, implies that \eqref{3.10} is impossible for large enough $m$, so we conclude that $\sgn (iz \Delta_m'(z))$ is eventually constant. Therefore, after possibly adjusting $m_0$, we may assume that $\sgn ( iz \Delta_m'(z))$ is constant for all $m\ge m_0$ and $z\in  I$; combining this with \eqref{3.9} gives \eqref{3.6}.

The lower bound in \eqref{3.7} is proved analogously to \eqref{3.6}, using Theorem~\ref{T2.2}(iv) and reality of $i c_m(z)$ on $I$ (by Theorem~\ref{T2.2}(iii)).
\end{proof}

\section{Diagonalization of transfer matrices}\label{S4}

We will start with a closed arc $I\subset\partial\mathbb{D}$ such that \eqref{1.13} holds. Following Lemma~\ref{L3.3}, we pick $m_0 \in \mathbb{N}_0$, $\epsilon>0$, $s, t \in \{-1,+1\}$ such that \eqref{3.5}, \eqref{3.6}, \eqref{3.7} hold on $\Omega$ given by \eqref{3.8}.

The goal of this section is to diagonalize the $\tilde \Phi_m(z)$ for $m\ge m_0$ and $z\in \Omega$ in a way which obeys the necessary uniform estimates in $z$ and $m$. Our first lemma provides uniform estimates on solutions of $\lambda^2 - \Delta \lambda + 1=0$. The second lemma uses this to produce uniform estimates for eigenvalues of $\tilde\Phi_m(z)$.

\begin{lemma}\label{L4.1}
For $\lvert \Delta \rvert < 2$, let
\[
\lambda_\pm (\Delta) = \frac {\Delta \pm i \sqrt{4-\Delta^2}}2
\]
be the solutions of $\lambda^2 - \Delta \lambda + 1=0$, taking the branch of $\sqrt{~}$ on $\mathbb{C} \setminus (-\infty,0]$ such that $\sqrt 1 = 1$. For any $\epsilon >0$, there is a value of $C >0$ such that:
\begin{enumerate}[(i)]
\item $\lvert \Delta \rvert \le 2 - \epsilon$ and  $\Im \Delta \ge 0$ implies
\[
\lvert \lambda_+(\Delta) \rvert \ge 1 + C \Im \Delta, \qquad \lvert \lambda_-(\Delta) \rvert \le  1 - C \Im \Delta;
\]
\item   $\lvert \Delta \rvert \le 2 - \epsilon$ and $\Im \Delta \le 0$ implies
\[
\lvert \lambda_+(\Delta) \rvert \le 1 + C \Im \Delta, \qquad \lvert \lambda_-(\Delta) \rvert \ge  1 - C \Im \Delta;
\]
\item  $\lvert \Delta \rvert \le 2 - \epsilon$ implies
\[
\Im \lambda_+(\Delta) \ge C, \qquad \Im \lambda_-(\Delta) \le -C.
\]
\end{enumerate}
\end{lemma}

\begin{proof}
Denote $\Delta = x + i y$ with $x,y\in\mathbb{R}$. Then $\frac{\partial \Delta}{\partial y}=i$,  $\frac{\partial \bar\Delta}{\partial y}=-i$, so we compute
\[
\frac{\partial}{\partial y} \lvert \lambda_+ \rvert^2 =  \frac{\partial}{\partial y} (\bar \lambda_+ \lambda_+) 
 = \frac 1{2} \Re \left((\bar \Delta - i \sqrt{4-\bar\Delta^2})(i + \frac{\Delta}{\sqrt{4-\Delta^2}} ) \right)
\]
which becomes
\[
\frac{\partial}{\partial y} \lvert \lambda_+ \rvert^2  =  \frac {y (\lvert 4 - \Delta^2\rvert + 4 + \lvert \Delta\rvert^2) + (\lvert \Delta\rvert^2 + \lvert 4-\Delta^2\rvert) \Re \sqrt{4-\Delta^2} }{ 2 \lvert 4-\Delta^2 \rvert} > 0
\]
when $\lvert \Delta \rvert <2$ and $y\ge 0$. Using $\lambda_+ \lambda_- = 1$, this implies
\begin{align*}
\frac{\partial}{\partial y} \lvert \lambda_+ \rvert & = \frac 1{2\lvert \lambda_+ \rvert} \frac{\partial}{\partial y} \lvert \lambda_+ \rvert^2 > 0 \\
\frac{\partial}{\partial y} \lvert \lambda_- \rvert & = - \frac 1{2\lvert \lambda_+ \rvert^3} \frac{\partial}{\partial y} \lvert \lambda_+ \rvert^2  < 0
\end{align*}
when $y\ge 0$. Continuity and compactness imply that for some $C>0$,
\[
\frac{\partial}{\partial y} \lvert \lambda_+ \rvert \ge C, \qquad 
\frac{\partial}{\partial y} \lvert \lambda_- \rvert \le -C
\]
uniformly in $\Delta$ with  $\lvert \Delta\rvert \le 2 -\epsilon$ and $y\ge 0$. Integrating in $y$ and using $\lvert \lambda_\pm (\Delta)\rvert = 1$ for $\Delta \in (-2,2)$ implies (i).

(ii) follows from (i) and $\lambda_\pm(\bar\Delta) = \overline{\lambda_\mp(\Delta)}$.

Note that $\lambda_+ \in \mathbb{R}$ would imply $\lambda_-  = \frac 1{\lambda_+} \in \mathbb{R}$ and $\lvert \Delta \rvert = \lvert \lambda_+ + \frac 1{\lambda_+} \rvert \ge 2$, which is a contradiction. Continuity and $\lambda_+(0) = i =  - \lambda_-(0)$ then imply
\[
\Im \lambda_+(\Delta) > 0 > \Im \lambda_-(\Delta)
\]
for $\lvert \Delta \rvert \le 2 - \epsilon$. By continuity and compactness, (iii) holds for some $C>0$.
\end{proof}

\begin{remark}
A part of the above calculations could have been skipped by only computing $\frac{\partial}{\partial y} \lvert \lambda_+ \rvert^2$ for $y=0$ and restricting the lemma to $\lvert \Im \Delta \rvert \le \epsilon$ for some $\epsilon$. However, to apply that to $\Delta_m$, we would then need a uniform upper bound for $\Im \Delta_m(z)$ in what follows. We chose instead to prove Lemma~\ref{L4.1} in more generality.
\end{remark}

We use the above lemma to choose an eigenvalue of $\tilde\Phi_m(z)$ in a consistent way:

\begin{lemma}\label{L4.2} With $s$ as in \eqref{3.6}, define
\[
\lambda_{m}(z) = \begin{cases} \lambda_+(\Delta_m(z)) & s = +1 \\
\lambda_-(\Delta_m(z)) & s = -1
\end{cases}
\]
Then $\lambda_m(z)$ and $\lambda^{-1}_m(z)$ are the eigenvalues of $\tilde \Phi_m(z)$, and they obey the following estimates for some $C>0$, uniformly in $m\ge m_0$, $z\in \Omega$:
\begin{equation}\label{4.1}
C \le s \Im \lambda_{m}(z) \le \lvert \lambda_{m}(z) \rvert \le 1 -  C (1 - \lvert z \rvert)
\end{equation}
\begin{equation}\label{4.2}
s \Im \lambda^{-1}_{m}(z) \le -C.
\end{equation}
\end{lemma}

\begin{proof}
$\lambda_{m}(z)$  and $\lambda^{-1}_{m}(z)$ are eigenvalues of $\tilde\Phi_m(z)$ since $\det\tilde\Phi_m(z)=1$ and $\tr\tilde\Phi_m(z)=\Delta_m(z)$.
Note that
\[
\frac{\partial}{\partial r} \Delta_m(r e^{i\theta}) =  e^{i\theta} \Delta'_m(r e^{i\theta})
\]
so, taking imaginary parts and multiplying by $s$,
\[
s \frac{\partial}{\partial r} \Im  \Delta_m(r e^{i\theta}) = \frac sr \Im \left( r e^{i\theta} \Delta'_m(r e^{i\theta})    \right) \ge C
\]
for some $C>0$ independent of $m$ and $z$, by \eqref{3.6}. Integrating in $r$, together with $\Im \Delta_m(e^{i\theta}) = 0$, gives
\[
-s \Im \Delta_m(r e^{i\theta}) = \int_r^1 s \frac{\partial}{\partial r} \Im  \Delta_m( t e^{i\theta}) d t  \ge C (1 -r).
\]
Combining this with  Lemma~\ref{L4.1}(i),(ii) implies the upper bound in \eqref{4.1} (with a different value of $C>0$). The bounds on $s \Im \lambda_m^{\pm 1}(z)$ follow from Lemma~\ref{L4.1}(iii).
\end{proof}

We wish to diagonalize $\tilde \Phi_m$ as
\[
\tilde\Phi_m(z) = U_m(z) \Lambda_m(z) U_m(z)^{-1}, \qquad \Lambda_m(z) = \begin{pmatrix} \lambda_{m}(z) & 0 \\ 0 & \lambda^{-1}_{m}(z) \end{pmatrix}
\]
so columns of $U_m$ should be eigenvectors of $\tilde\Phi_m$. We choose $U_m(z)$ as
\begin{equation}\label{4.3}
U_m(z) = \begin{pmatrix} \lambda_{m}(z) - d_m(z) &  \lambda^{-1}_{m}(z) - d_m(z) \\  c_m(z) & c_m(z) \end{pmatrix}.
\end{equation}
The determinant of $U_m(z)$ is
\[
\det U_m = (\lambda_{m} - \lambda^{-1}_{m}) c_m,
\]
which is non-zero since in the region of interest, $c_m(z) \neq 0$ and $\lambda_{m}(z) \neq \lambda^{-1}_{m}(z)$ (this follows from \eqref{3.7} and Lemma~\ref{L4.2}). We also compute
\begin{equation}\label{4.4}
U_m^{-1} = \frac 1{(\lambda_{m} - \lambda^{-1}_{m}) c_m} \begin{pmatrix} c_m  & d_m - \lambda^{-1}_{m} \\ - c_m & \lambda_{m} - d_m \end{pmatrix}
\end{equation}
and we define
\[
W_m = U_m^{-1} U_{m+1} - I.
\]

\begin{lemma}\label{L4.3}
For some value of $C<\infty$, uniformly in $m\ge m_0$ and $z\in \Omega$ we have
\begin{equation}\label{4.5}
\lVert W_m \rVert \le C \sum_{k=0}^{p-1} \lvert \alpha_{(m+1)p+k} - \alpha_{mp+k} \rvert.
\end{equation}
\end{lemma}

\begin{proof}
From the definitions of $\lambda_{m}(z)$ and $U_m(z)$ and Lemma~\ref{L3.1}, it is clear that $\lambda_{m}(z)$ and $U_m(z)$ have the same properties listed in that lemma. Therefore, for some $C < \infty$, we have $\lVert U_m^{-1} \rVert \le C$ and
\[
\lVert U_{m+1} - U_m \rVert \le C \sum_{k=0}^{p-1} \lvert \alpha_{(m+1)p+k} - \alpha_{mp+k} \rvert.
\]
Now \eqref{4.5}, with a different $C$, follows from $W_m = U_m^{-1} (U_{m+1} - U_m)$.
\end{proof}

\section{Approximants and Weyl solutions}\label{S5}

In this section, we carry over an idea of Kaluzhny--Shamis~\cite{KaluzhnyShamis12} to introduce approximants to $\mu$ with eventually periodic sequences of coefficients, and relate their absolutely continuous parts to certain Weyl solutions.

Define the measure $\mu^N$, $N=0,1,\dots$, so that its first $(N+1)p$ Verblunsky coefficients agree with those of $\mu$, and extending the sequence by $p$-periodicity after that; i.e., the Verblunsky coefficients of $\mu^N$ are
\begin{equation}\label{5.1}
\alpha^N_{mp+r} = \begin{cases} \alpha_{mp+r} & m<N, r = 0,1,\dots, p-1  \\
\alpha_{Np+r} & m\ge N, r=0,1,\dots, p-1
\end{cases}
\end{equation}
We will also denote other quantities corresponding to $\mu^N$ with the superscript $N$; for instance, the $p$-step transfer matrices corresponding to $\mu^N$ are, by \eqref{5.1}, $ \Phi^N_m = \Phi_{\min(N,m)}$, and the modified transfer matrices are
\[
\tilde \Phi^N_m(z) = \tilde \Phi_{\min(N,m)}(z).
\]
For $N\ge m_0$, we now single out a solution $u^N(z)$ of the recursion relation
\[
u^N_{n+1} (z)= \tilde \Phi^N_n (z) u^N_n(z).
\]
Since all $\tilde \Phi_n$ are invertible, we can specify the solution by setting its value at $n=N$,
\begin{equation}\label{5.2}
u^N_N(z) = \begin{pmatrix} \lambda_{N}(z) - d_N(z) \\ c_N(z) \end{pmatrix}.
\end{equation}
Let $\mu^N$ have the Lebesgue decomposition
\[
d\mu^N = w^N \frac{d\theta}{2\pi} + d\mu^N_\s. 
\]
We will now describe $w^N$ in terms of $u^N$.

\begin{lemma}\label{L5.1}
Let $N\ge m_0$. For every $z\in I$, $(u_0^N)_2(z) \neq 0$. For Lebesgue-a.e.\ $z\in I$,
\begin{equation}\label{5.3}
w^N(z) =  - \frac{ i c_N(z) \Im \lambda_{N}(z) }{\lvert (u^N_0)_2(z) \rvert^2}.
\end{equation}
\end{lemma}

\begin{remark}
By Theorem~\ref{T2.2}(iii), we already know that the right hand side of \eqref{5.3} is real-valued. In fact, comparing $w^N(z)\ge 0$ with \eqref{3.7} and \eqref{4.1} gives $s=t$. This observation will not be needed in what follows.
\end{remark}

\begin{proof}
For $\lvert z \rvert =1$, the relation $A(\alpha,z)^* J A(\alpha,z) = J$ holds for all $\alpha\in\mathbb{D}$, where
\[
J = \begin{pmatrix} 1 & 0 \\ 0 & -1 \end{pmatrix}.
\]
This implies $\Phi_n(z)^* J \Phi_n(z) = J$, and then $\tilde\Phi_n(z)^* \tilde J \tilde \Phi_n(z) = \tilde J$, where
\[
\tilde J = M J M = \begin{pmatrix} 0 & 1 \\ 1 & 0 \end{pmatrix}.
\]
This implies constancy of the Wronskian in the form
\[
\langle u^N_0(z), \tilde J u^N_0(z) \rangle = \langle u^N_N(z), \tilde J  u^N_N(z) \rangle,
\]
which simplifies to
\[
2 \Re ( (u^N_0)_1(z) \overline{ (u^N_0)_2(z)}) = 2 \Re ( (u^N_N)_1(z) \overline{ (u^N_N)_2(z)} ).
\]
Using \eqref{5.2} and Theorem~\ref{T2.2}(iii), this simplifies to
\begin{equation}\label{5.4}
2 \Re ( (u^N_0)_1(z) \overline{ (u^N_0)_2(z)}) = - 2 i c_N(z) \Im \lambda_{N}(z).
\end{equation}
In particular, by \eqref{3.7} and \eqref{4.1}, this implies that $(u^N_0)_1(z) \overline{ (u^N_0)_2(z)} \neq 0$ for $z\in I$.

From  $\tilde \Phi^N_n u^N_N = \lambda_{N} u^N_N$ for $n\ge N$ and $\lvert \lambda_{N} \rvert < 1$ it follows that $z^{nq/2} M u^N_n$ is a Weyl solution for $\lvert z \rvert < 1$ (see \cite[Section 2.3]{Rice} for definition and properties). However, recall that
\[
v_n = \Phi^N_{n-1} \cdots \Phi^N_{0} \begin{pmatrix} 1 \\ z f^N(z) \end{pmatrix}
\]
is also a Weyl solution for $\lvert z \rvert <1$, where $f^N$ is the Schur function for $\mu^N$. The Caratheodory function for $\mu^N$ is
\[
F^N(z) = \frac{1+z f^N(z)}{1-zf^N(z)},
\]
which we rewrite as
\[
M\begin{pmatrix} 1 \\ z f^N(z) \end{pmatrix} = \frac 1{\sqrt 2} \begin{pmatrix} 1 + z f^N(z) \\ 1- z f^N(z) \end{pmatrix} = \frac {1-z f^N(z)}{\sqrt 2} \begin{pmatrix} F^N(z) \\ 1 \end{pmatrix}.
\]
Since Weyl solutions are unique up to a multiplicative constant, we conclude that $\begin{pmatrix} F^N(z) \\ 1 \end{pmatrix}$ is a multiple of $u_0^N$, so
\[
F^N(z) = \frac{(u^N_0)_1(z)}{(u^N_0)_2(z)}.
\]
For almost every $z\in \partial \mathbb{D}$, the nontangential limit of $\Re F^N(z)$ is equal to $w^N(z)$, so
\[
w^N(e^{i\theta}) = \lim_{r\uparrow 1} \Re \frac{(u^N_0)_1(r e^{i\theta})}{(u^N_0)_2(r e^{i\theta})}.
\]
The limit exists for all $e^{i\theta} \in I$ because $u^N_N$, and so $u^N_n$ for every $n$, is continuous in $z \in \Omega$. Using \eqref{5.4}, this simplifies to \eqref{5.3}.
\end{proof}

\section{Conclusion of the proof of Theorem~\ref{T1.4}}\label{S6}

In this section, we carry over the method of Denisov and Kaluzhny--Shamis~\cite{Denisov09, KaluzhnyShamis12} to OPUC, with the modifications necessary to handle the lack of asymptotic convergence.

Coefficient stripping is the process of removing the leading Verblunsky coefficient, i.e.\ replacing a measure $\mu$ with Verblunsky coefficients $\alpha$ by the measure $\mu_1$ with Verblunsky coefficients $S\alpha$. It is well known that this operation does not affect the validity of conclusions of Theorem~\ref{T1.4}; for instance, this follows from properties of the relative Szeg\H o function \cite[Theorem 2.6.2]{Rice}.

We can use this to perform coefficient stripping finitely many times and prove the result for the measure obtained in this way, from which the result for the original measure will follow. Thus, in the following we may assume that all the above estimates, derived for $m\ge m_0$, now hold for all $m\ge 0$. By additional coefficient stripping, we may also assume that
\begin{equation}\label{6.1}
\sum_{n=0}^\infty \lVert W_n\rVert^2 < \delta
\end{equation}
for a suitably chosen $\delta >0$ (to be chosen later).

The recursion relation for $u^N_n$, solved backwards, gives
\[
u^N_0 = \tilde \Phi_0^{-1}  \cdots \tilde \Phi_{N-1}^{-1} u^N_N.
\]
Using the diagonalization of $\tilde\Phi_n$, this becomes
\[
u^N_0 = U_0 \Lambda_0^{-1} U^{-1}_0 \cdots  U_{N-1} \Lambda_{N-1}^{-1}  U_{N-1}^{-1} u^N_N.
\]
A direct calculation shows $U_N^{-1} u^N_N = \begin{pmatrix} 1\\ 0 \end{pmatrix}$, so

\begin{equation}\label{6.2}
U_0^{-1} u^N_0 =   \Lambda_0^{-1} (I + W_0) \cdots \Lambda_{N-1}^{-1} (I + W_{N-1}) \begin{pmatrix} 1\\ 0 \end{pmatrix}.
\end{equation}

We will now need a lemma of Denisov~\cite{Denisov09}, made precisely to estimate such products.

\begin{thm}[{\cite[Theorem 2.1]{Denisov09}}] \label{T6.1}
Let
\begin{equation}\label{6.3}
\Psi_{n+1} =  \begin{pmatrix}  \kappa_n & 0 \\ 0 &  \kappa_n^{-1} \end{pmatrix}  (I + W_n) \Psi_n, \quad W_n = \begin{pmatrix} e_n & f_n \\ g_n & h_n \end{pmatrix}, \quad \Psi_0 = \begin{pmatrix} 1 \\ 0 \end{pmatrix}
\end{equation}
where $\kappa_n \in \mathbb{C}$,
\begin{equation}\label{6.4}
C > \lvert \kappa_n \rvert >  \kappa >1,
\end{equation}
and the sum $\sum_{n=0}^\infty \lVert W_n \rVert^2$ is finite and sufficiently small. Assume also there is a constant $v \in [0,1)$ such that
\[
\left\lvert \log \prod_{n=k}^l \lvert 1 + e_n \rvert \right\rvert \le C + C v \sqrt{l-k}, \qquad \left\lvert \log \prod_{n=k}^l \lvert 1 + h_n \rvert \right\rvert \le C + C v \sqrt{l-k}.
\]
Then there is a value of $C_1 \in (0,\infty)$, which depends only on $C$, such that
\[
\Psi_n = \prod_{j=0}^{n-1} \left( \kappa_j (1+e_j) \right) \begin{pmatrix} \phi_n \\ \nu_n \end{pmatrix}
\]
where
\[
\lvert \phi_n \rvert, \lvert \nu_n \rvert \le C_1 \exp \left( \frac {C_1}{ \kappa -1} \exp\left( \frac{C_1 v^2}{\kappa - 1 } \right)  \right)
\]
Moreover, for any fixed $\epsilon>0$ and $\kappa> 1+\epsilon$, we have
\[
\lvert \phi_n \rvert > C_1^{-1} > 0, \qquad \lvert \nu_n \rvert < C_1 \sum_{j=0}^\infty \lVert W_j \rVert^2
\]
uniformly in $n$ provided that $\sum_{j=0}^\infty \lVert W_j \rVert^2$ is small enough.
\end{thm}

\begin{remark}
Compared to \cite{Denisov09}, we have switched the order of $\begin{pmatrix}  \kappa_n & 0 \\ 0 &  \kappa_n^{-1} \end{pmatrix}$ and $I+W_n$ in \eqref{6.3}; this is better suited to our use. This can be proved with minimal modifications to the proof in \cite{Denisov09}. Alternatively, by inserting an additional $I+W_N = I$ and  $\kappa_{-1} = \frac{C+\kappa}2$, \eqref{6.3} can be rewritten as
\[
\Psi_{n} = \kappa_{-1}^{-1} (I+W_n) \begin{pmatrix} \kappa_{n-1} & 0 \\ 0 &  \kappa_{n-1}^{-1} \end{pmatrix}   \dots (I+W_0) \begin{pmatrix} \kappa_{-1} & 0 \\ 0 &  \kappa_{-1}^{-1} \end{pmatrix} \begin{pmatrix} 1 \\ 0 \end{pmatrix},
\]
in which we can group $I+W_j$ with $\begin{pmatrix} \kappa_{j-1} & 0 \\ 0 &  \kappa_{j-1}^{-1} \end{pmatrix}$ and apply the  version stated in \cite{Denisov09}. 
\end{remark}

In order to apply Theorem~\ref{T6.1} to \eqref{6.2}, we now verify that conditions of Theorem~\ref{T6.1} are met. Our $\kappa_n = \lambda_n^{-1}$, so \eqref{6.4} follows from \eqref{4.1}. From \eqref{4.3} and \eqref{4.4} we compute
\begin{align*}
1+ e_n & =  \frac {c_n(\lambda_{n+1} - d_{n+1}) + c_{n+1} (d_n - \lambda^{-1}_{n})}{(\lambda_{n} - \lambda^{-1}_{n}) c_n}, \\
1+ h_n & = \frac {-c_n(\lambda^{-1}_{n+1} - d_{n+1}) + c_{n+1}(\lambda_{n} - d_n)} {(\lambda_{n} - \lambda^{-1}_{n}) c_n}.
\end{align*}
Then
\begin{equation}\label{6.5}
\left\lvert \log \prod_{n=k}^l \lvert 1 + e_n \rvert \right\rvert \le C + C(1 - \lvert z \rvert)\sqrt{l-k}
\end{equation}
(and the same inequality with $h_n$ instead of $e_n$) is proved almost as in the proof of Theorem 2.2 of \cite{Denisov09}; a modification is needed where \cite{Denisov09} uses convergence of coefficients, so Lemma~2.5 of \cite{Denisov09} must be replaced by
\begin{lemma}\label{L6.2}
If $\{\epsilon_n\}_{n=0}^\infty$ is a sequence of complex numbers and $C<\infty$ a constant such that for all $n$,
\begin{equation}\label{6.6}
C^{-1} \le \Im \epsilon_n \le  \lvert \epsilon_n \rvert \le C,
\end{equation}
and
\begin{equation}\label{6.7}
\sum_n \lvert \epsilon_{n+1} - \epsilon_n \rvert^2 \le C,
\end{equation}
then there is a constant $C_1<\infty$ which depends only on $C$, such that for all $k\le l$,
\[
\left\lvert \sum_{n=k}^l \frac{\epsilon_{n+1} - \epsilon_n}{\epsilon_n} \right\rvert < C_1.
\]
\end{lemma}

\begin{proof}
Let us fix branches of $\log$ and $\arg$ on $\mathbb{C} \setminus (-\infty,0]$ with
\[
\Im \log = \arg \in (-\pi,\pi).
\]
The assumptions of the lemma imply that $\frac{\epsilon_{n+1}}{\epsilon_n} \in S$ for all $n$, where
\[
S = \{ z \in \mathbb{C} \mid C^{-2} \le \lvert z \rvert \le C^2,   \lvert \arg z \rvert \le \pi -  2 \arcsin(C^{-2}) \}.
\]
Compactness of $S\subset \mathbb{C} \setminus (-\infty,0]$  and analyticity of $\frac{z - 1 - \log z}{(z-1)^2}$ in $\mathbb{C} \setminus (-\infty,0]$ imply that for some $C_2<\infty$ and all $z\in S$,
\[
\lvert z -1 - \log z \rvert \le C_2 \lvert z-1 \rvert^2.
\]
Applying this to $z = \frac{\epsilon_{n+1}}{\epsilon_n}$ and summing in $n$, we conclude
\begin{equation}\label{6.8}
\left \lvert \sum_{n=k}^l \left( \frac{\epsilon_{n+1} - \epsilon_n}{\epsilon_n} -  \log \frac{\epsilon_{n+1}}{\epsilon_n} \right) \right\rvert  \le C_2 \sum_{n=k}^l \left\lvert \frac{\epsilon_{n+1} - \epsilon_n}{\epsilon_n}  \right\rvert^2.
\end{equation}
Since $\Im \epsilon_n > 0$ for all $n$, with our choice of branches we have
\[
\arg \frac{\epsilon_m}{\epsilon_n} = \arg \epsilon_m - \arg \epsilon_n
\]
for any $m,n$, and so
\[
\sum_{n=k}^l \log \frac{\epsilon_{n+1}}{\epsilon_n} = \log \frac{\epsilon_{l+1}}{\epsilon_k}.
\]
Thus, \eqref{6.8} and $\lvert \epsilon_n \rvert \ge C^{-1}$ implies
\[
\left \lvert \sum_{n=k}^l  \frac{\epsilon_{n+1} - \epsilon_n}{\epsilon_n}  -  \log \frac{\epsilon_{l+1}}{\epsilon_k} \right\rvert  \le  \frac{C_2}{C^2} \sum_{n=0}^\infty \left\lvert \epsilon_{n+1} - \epsilon_n \right\rvert^2. 
\]
The proof is completed by noting that $ \left\lvert \log \frac{\epsilon_{l+1}}{\epsilon_k} \right\rvert$ is uniformly bounded in $k, l$ by \eqref{6.6} and using \eqref{6.7}.
\end{proof}

Following \cite{Denisov09}, Lemma~\ref{L6.2} is applied to $\epsilon_n = t c_n$ and $\epsilon_n = s( \lambda_n - \lambda_n^{-1})$. They obey all the conditions by \eqref{3.7}, \eqref{4.1}, \eqref{4.2}, and \eqref{3.2}.

Thus, Theorem~\ref{T6.1} is applicable to \eqref{6.2} with $\kappa = 1 + C (1-\lvert z\rvert)$ and $v = 1-\lvert z\rvert$, and we conclude that
\begin{equation}\label{6.9}
U_0^{-1}(z) u_0^N(z) = \prod_{n=1}^N (\lambda_n^{-1}(z) (1+e_n(z))) \begin{pmatrix} \phi_N(z) \\ \nu_N(z) \end{pmatrix}
\end{equation}
with $\phi_N$, $\nu_N$ which obey, since $v^2/(\kappa-1)$ is uniformly bounded for $z\in \Omega$,
\begin{equation}\label{6.10}
\lvert \phi_N \rvert, \lvert \nu_N \rvert \le \exp\left( \frac{C}{1 - \lvert z\rvert} \right) 
\end{equation}
for some $C<\infty$ and all $N$ and $z\in \Omega$. Moreover, if $\delta$ in \eqref{6.1} has been chosen small enough, then
\begin{equation}\label{6.11}
\lvert \phi_N \rvert > C, \qquad \lvert \nu_N \rvert < \frac C2, \qquad \text{for $z\in\Omega$ with $1-\lvert z \rvert> \frac \epsilon 2$}.
\end{equation}

Multiplying \eqref{6.9} by $U_0(z)$ and using \eqref{4.3}, we see
\begin{equation}\label{6.12}
(u_0^N)_2(z) = \prod_{n=1}^N \left(\lambda_n^{-1}(z) (1+e_n(z))\right) c_0(z) (\phi_N(z) + \nu_N(z) )
\end{equation}
which we rewrite as
\begin{equation}\label{6.13}
-\log \lvert (u_0^N)_2(z)\rvert = - \log \prod_{n=1}^N  \left \lvert \lambda_n^{-1}(z) (1+e_n(z)) \right\rvert - \log \lvert c_0(z) \rvert  + f_N(z)
\end{equation}
where
\[
f_N(z) =  - \log \left\lvert \phi_N(z) + \nu_N(z) \right\rvert.
\]

\begin{lemma} \label{L6.3}
The function $f_N(z)$ is continuous on $\Omega$ and harmonic on $\Int \Omega$. There is a value of $C\in (0,\infty)$, independent of $N\in \mathbb{N}_0$, such that
\begin{enumerate}[(i)]
\item for all $z\in I$ and $N\in \mathbb{N}_0$,
\begin{equation}\label{6.14}
\left\lvert  \log  w^N(z)   - 2 f_N(z)  \right\rvert  \le C
\end{equation}
\item for all $N\in \mathbb{N}_0$,
\begin{equation}\label{6.15}
\int_I f_N^+(z) \frac{d\theta}{2\pi} \le C
\end{equation}
\item for all $z\in \Omega\setminus I$ and $N\in \mathbb{N}_0$,
\begin{equation}\label{6.16}
f_N(z) \ge - \frac C{1 - \lvert z \rvert}
\end{equation}
\item for all $z\in \Omega$ with $1-\lvert z \rvert > \tfrac 12 \epsilon$ (this is $\epsilon$ from \eqref{3.8}) and $N\in \mathbb{N}_0$,
\begin{equation}\label{6.17}
f_N(z) \le C.
\end{equation}
\end{enumerate}
\end{lemma}

\begin{proof}
For $z\in\Omega$, $\phi_N(z)+\nu_N(z) \neq 0$ by \eqref{6.12} and Lemma~\ref{L5.1}. Moreover, $\phi_N(z) + \nu_N(z)$ are analytic in $z$ by \eqref{6.9}, so the same is true of $-\log(\phi_N(z) + \nu_N(z))$. Since  $f_N(z) = - \Re \log (\phi_N(z)+ \nu_N(z))$, the conclusions hold.

For $z\in I$, using $\lvert \lambda_n(z) \rvert=1$ and  \eqref{6.13}, we can rewrite \eqref{5.3} as
\[
\log \lvert w^N(z) \rvert  = \log \lvert c_N(z)\rvert + \log \lvert \Im\lambda_N(z)\rvert  - 2 \log \prod_{n=1}^N  \left \lvert 1+e_n(z) \right\rvert - 2 \log \lvert c_0(z) \rvert + 2 f_N(z)
\]

For $z\in I$, $\prod_{n=1}^N \lvert 1+e_n(z) \rvert$ is uniformly bounded by \eqref{6.5} and $\log\lvert c_0(z)\rvert$, $\log\lvert c_N(z)\rvert$  by \eqref{3.7}, which proves \eqref{6.14}.

Using $\log  w^N(z) \le w^N(z) -1$ and the fact that $w^N$ is the a.c.\ part of a probability measure,
\[
\int (\log  w^N(z))^+ \frac{d\theta}{2\pi} \le \int_I w^N(z) \frac{d\theta}{2\pi} \le \mu^N(I) \le 1.
\]
With \eqref{6.14}, this implies \eqref{6.15}.

\eqref{6.16} follows from \eqref{6.10}, and \eqref{6.17} follows from \eqref{6.11} and
\[
\left\lvert \phi_N(z) + \nu_N(z) \right\rvert \ge  \lvert \phi_N(z) \rvert - \lvert \nu_N(z) \rvert. \qedhere
\]
\end{proof}

\begin{lemma}[{\cite{Denisov09,KaluzhnyShamis12}}] \label{L6.4}
Assume that $f(z)$ is continuous on $\Omega$, harmonic on $\Int \Omega$, and for some $C, \alpha >0$,
\[
\int_I f^+(e^{i\theta}) \frac{d\theta}{2\pi} < C,
\]
$f(z) > - C (1-\lvert z\rvert)^{-\beta}$ for $z\in \Int\Omega$, and $f(z)<C$ for $z\in\Omega$ with $1-\lvert z\rvert > \frac{C}{1+\beta}$. Then there is a constant $B$, depending only on $C,\beta$, so that
\[
\int_I f^-(e^{i\theta}) \frac{d\theta}{2\pi} <B.
\]
\end{lemma}

In the given references, this is a lemma on a interval $I$ on $\mathbb{R}$, rather than an arc on $\partial\mathbb{D}$. Using a conformal map which maps the unit disk to the upper half plane, the statement given here is an immediate corollary of \cite[Lemma 2]{KaluzhnyShamis12}.

By Lemma~\ref{L6.3}, Lemma~\ref{L6.4} is applicable to $f_N(z)$, and proves
\[
\int_I f_N(e^{i\theta}) \frac{d\theta}{2\pi} < C
\]
with a constant $C$ independent of $N$. By \eqref{6.15} and \eqref{6.14}, this implies
\[
\int_I \log w^{N}(e^{i\theta}) \frac{d\theta}{2\pi} > C
\]
with $C\in \mathbb{R}$ independent of $N$.

This integral is a relative entropy: in the notation of \cite[Section 2.2]{Rice}, with $\chi_I$ the characteristic function of $I$,
\[
\int_I \log w^{N}(e^{i\theta}) \frac{d\theta}{2\pi} = S \left( \chi_I \frac{d\theta}{2\pi} \big\vert \chi_I d\mu^N \right).
\]
Since $\alpha^N$ converge pointwise to $\alpha$, the measures $\mu^N$ converge weakly to $\mu$, so upper semicontinuity of entropy \cite[Theorem 2.2.3]{Rice} gives
\[
\int_I \log w(e^{i\theta}) \frac{d\theta}{2\pi} \ge \limsup_{N\to\infty}  \int_I \log w^{N}(e^{i\theta}) \frac{d\theta}{2\pi} \ge C > -\infty
\]
which proves \eqref{1.7}.

\eqref{1.7} implies that  for a.e. $e^{i\theta} \in I$, $\log w(e^{i\theta}) > -\infty$, and thus $w(e^{i\theta}) > 0$. This implies that $I \subset \esssupp w$. Since $L$ is continuous, for every $z\in \partial\mathbb{D}$ we may find a suitable arc $I$ which contains it, so $z \in I \subset  \Sigma_\ac(\mu)$. This proves the first inclusion in \eqref{1.14b}.

By the Last--Simon \cite{LastSimon99} theorem for a.c.\ spectrum of right limits (extended to OPUC by Simon \cite[Theorem 10.9.11]{OPUC2}), for any right limit $\alpha^{(r)}$,
\[
\Sigma_\ac(\mu)  \subset \{ z \in \partial\mathbb{D} \mid \lvert \Delta^{(r)}(z) \rvert \le 2\}.
\]
By \eqref{1.12}, taking the intersection over all right limits proves the second inclusion in \eqref{1.14b} and completes the proof of Theorem~\ref{T1.4}.

\section{Comparing the lower and upper bounds on $\Sigma_\ac(\mu)$}\label{S7}

Theorem~\ref{T1.4} gives lower and upper bounds on $\Sigma_\ac(\mu)$. In this section, we explore cases in which the lower and upper bounds coincide.

Equality of the sets in \eqref{1.14} follows from a mild condition:

\begin{lemma}\label{L7.1}
If \eqref{1.5} holds and $\{ z \in \partial\mathbb{D} \mid L(z) = 2\}$ has zero Lebesgue measure, then all sets in \eqref{1.14} are equal.
\end{lemma}

\begin{proof}
By Lemma~\ref{L3.2}, $L(z)$ is continuous, so the set $Y = \{ z \in \partial\mathbb{D} \mid L(z) < 2\}$ is open and $\overline{Y} = \overline{Y}^\ess$.  If the set $X = \{ z \in \partial\mathbb{D} \mid L(z) = 2\}$ has Lebesgue measure $0$, then $\overline{Y}^\ess = \overline{X\cup Y}^\ess$. Thus, $\overline Y = \overline{X \cup Y}^\ess$ and equality of the sets in \eqref{1.14} follows. 
\end{proof}

In all our applications, $\{ z \in \partial\mathbb{D} \mid L(z) = 2\}$ will be a finite set.

\begin{proof}[Proof of Theorem~\ref{T1.1}] 
A straightforward calculation  together with \eqref{3.4} gives
\[
\lvert \Delta_m(e^{i\theta}) \rvert = \frac {2 \lvert \cos(\theta/2) \rvert}{\rho_m}, \qquad
L(e^{i\theta}) = \frac{2 \lvert \cos(\theta/2)\rvert}{\sqrt{1 - A^2 }},
\]
so $L(e^{i\theta})=2$ is equivalent to $\lvert \cos(\theta/2)\rvert = \sqrt{1-A^2}$. This holds only on a finite set, so Lemma~\ref{L7.1} implies equality of all sets in \eqref{1.14}.

Moreover, $L(e^{i\theta})<2$ is equivalent to $\lvert \cos(\theta/2) \rvert < \sqrt{1-A^2}$  and to
\[
2 \arcsin A < \theta < 2 \pi - 2\arcsin A,
\]
so the claim follows from \eqref{1.14} and Theorem~\ref{T1.4}.
\end{proof}

\begin{proof}[Proof of Theorem~\ref{T1.2}]
A straightforward calculation together with \eqref{3.4} gives
\[
\Delta_m(e^{i\theta}) = 2 \frac{\cos \theta + \Re(\alpha_{2m} \bar \alpha_{2m+1})}{ \rho_{2m}\rho_{2m+1} } .
\]
Using uniform boundedness of $\rho_{2m}\rho_{2m+1}$ given by
\[
0 < 1 - (\sup_n \lvert \alpha_n\rvert)^2 \le  \rho_{2m}\rho_{2m+1} \le 1,
\]
it is then easy to see that $L(e^{i\theta})<2$ is equivalent to $- A_+ < \cos \theta < A_-$ and $L(e^{i\theta})=2$ equivalent to $\cos \theta \in \{- A_+ , A_-\}$. Thus, the set of $e^{i\theta}$ such that $L(e^{i\theta})=2$ is finite, so Lemma~\ref{L7.1} implies equality of all sets in \eqref{1.14}. If $I \subset \Int( \esssupp w \setminus \{-1,1\})$, it is clear from the above that $L(z) < 2$ for $z \in I$, so \eqref{1.7} follows by Theorem~\ref{T1.4}.
\end{proof}

\begin{proof}[Proof of Theorem~\ref{T1.3}]
Since all right limits have the same discriminant $\Delta_e(z)$,
\[
L(z) = \lvert \Delta_e(z)\rvert.
\]
Since $\Delta_e(z)$ is a nontrivial polynomial in $z^{1/2}$, the set $\{ e^{i\theta} \in\partial\mathbb{D} \mid \Delta_e(e^{i\theta})\in \{-2,2\}\}$ is finite, so Lemma~\ref{L7.1} implies \eqref{1.11}. \eqref{1.7} follows from Theorem~\ref{T1.4}. 
\end{proof}

\bibliographystyle{amsplain}

\providecommand{\bysame}{\leavevmode\hbox to3em{\hrulefill}\thinspace}
\providecommand{\MR}{\relax\ifhmode\unskip\space\fi MR }
\providecommand{\MRhref}[2]{%
  \href{http://www.ams.org/mathscinet-getitem?mr=#1}{#2}
}
\providecommand{\href}[2]{#2}

\end{document}